\documentclass{article}

\RequirePackage[OT1]{fontenc}
\RequirePackage{amsthm,amsmath,amssymb}
\RequirePackage[numbers]{natbib}
\RequirePackage[all]{xy}
\RequirePackage{hyperref}
\usepackage{color}
\usepackage{graphics}

\usepackage{graphicx}

\newtheorem{proposition}{Proposition}[section]
\newtheorem{theorem}{Theorem}[section]
\newtheorem{definition}{Definition}[section]
\newtheorem{remark}{Remark}[section]
\newtheorem{corollary}{Corollary}[section]
\newtheorem{lemma}{Lemma}[section]
\numberwithin{equation}{section}
\theoremstyle{plain}

\begin{document}
\begin{center}
\textbf{On general characterization of Young measures associated with Borel functions}
\end{center}
\begin{center}
\textbf{Andrzej Z. Grzybowski}\\

{\small\mbox{Institute of
 Mathematics, Czestochowa University of Technology,
 al. Armii Krajowej 21}, 42-200 Cz\c{e}stochowa, Poland\\
Email: azgrzybowski@gmail.com}\\
\textbf{Piotr Pucha{\l}a}\\

{\small\mbox{Institute of
 Mathematics, Czestochowa University of Technology,
 al. Armii Krajowej 21}, 42-200 Cz\c{e}stochowa, Poland\\
Email: piotr.puchala@im.pcz.pl, p.st.puchala@gmail.com}
\end{center}

\title{On general characterization of Young measures associated with Borel functions}

\begin{abstract}
We prove that the Young measure associated with a Borel function $f$ is a probability distribution of the random variable $f(U)$, where $U$ has a uniform distribution on the domain of $f$. As an auxiliary result, the fact that Young measures associated with simple functions are weak$^{\ast}$ dense in the set of Young measures associated with measurable functions is proved. Finally some examples of specific applications of the main result are presented with comments. 

\vspace{0.5cm}

\noindent\textbf{Keywords}: Young measures, characterization of measures, Borel functions, uniform distribution\\
\noindent\textbf{AMS Subject Classification}: 46N10; 46N30; 49M30; 60A10
\bigskip
\end{abstract}

\section[]{Introduction} One of the major problems in the calculus of variations is minimization of functionals which are bounded from below but do not attain their infima. If the minimized functional $\mathcal{J}$ is bounded, the direct method can be applied: there always exists a minimizing sequence for $\mathcal{J}$, that is a sequence $(u_n)$, $u_n\colon\mathbb{R}^d\to
\mathbb{R}^l$, $n\in\mathbb{N}$, such that $\lim\limits_{n\to\infty}\mathcal{J}(u_n)=
\inf\mathcal{J}$. Additionally, if $\mathcal{J}$ is coercive, $(u_n)$ is always bounded. However, if $\mathcal{J}$ does not attain its infimum then the elements of $(u_n)$ are functions of highly oscillatory nature. Moreover, weak$^{\ast}$ convergence in $L^{\infty}$ of $(u_n)$ to some function $u_0$ does not guarantee, that the sequence $(\varphi(u_n))$ of compositions of $u_n$ with continuous function $\varphi$ is weak$^{\ast}$ convergent in $L^{\infty}$. Indeed, in general it is not convergent  not only to $\varphi(u_0)$, but to any function with domain in $\mathbb{R}^d$. 

Laurence Chisholm Young introduced in \cite{Young} objects called by him 'generali\-zed curves', nowadays called 'Young measures'.   
These are the 'generalized li\-mits' of sequences of highly oscillating functions. The 'mature' form of Young's theorem has been proved by J.M.Ball in \cite{Ball} (see also \cite{Pedregal}, theorem 6.2). According to these theorems, we say that under their assumptions the considered sequences 'generate' appropriate Young measures. This approach is studied for example in \cite{Pedregal} in detail.

Alternatively, we can look at the Young measure as at object associated with \textit{any} measurable function defined on a nonempty, open, bounded subset $\varOmega$ of
$\mathbb{R}^d$ with values in a compact subset $K$ of $\mathbb{R}^l$. Such a conclusion can be derived from the theorem 3.6.1 in \cite{Roubicek}. Thank to this theorem it can be proved that the Young measure associated with a simple function is the convex combination of Dirac measures. These Dirac measures are concentrated at the values of the simple function under consideration while coefficients of the convex combinaton are proportional to the Lebesgue measure of the sets on which the respective values are taken on by the function; see \cite{Puchala} for details and more general result concernig simple method of obtaining explicit form of Young measures associated with oscillating functions. This method does not need advanced functional analytic methods: it is based on the change of variable theorem.

In this article we significantly generalize the above results. We prove a theorem providing general yet simple description of  Young measures associated with Borel functions. As a consequence, the theorem enables one to compute explicit formulae of probability density functions of the Young measures in many interesting cases. What is more, it also can be done without any sophisticated functional analytic apparatus. Since Young measures are widely used in many areas of theoretical and applied sciences (see for example \cite{Balder2}, \cite{Malek}, \cite{Muller}, \cite{Pedregal1}), our result provide a handy tool of obtaining their explicit form.

The main theorem of this article  states 
that the Young measure associated with any Borel function $f$ defined on the set $\varOmega\subset\mathbb{R}^d$ with posi\-tive Lebesgue measure $M$ and values in a compact set $K\subset\mathbb R^l$,  is in fact a probability distribution of a random variable $X=f(U)$, where $U$ is uniformly distributed on $\varOmega$. Before this, we prove a lemma corresponding to standard measure-theoretic result, that any Borel function is a pointwise limit of the appropriate sequence of simple functions. This is neither new nor most general result of this type (in \cite{Balder1} the result is stated for completely regular Souslin spaces), but it seems that simplicity of the proof of its special case that may be of interest in applications makes it worth mentio\-ning. Relying on this fact we  prove, that for any Borel function $f\colon\varOmega\to K$, its Young measure is the weak$^{\ast}$ limit of a sequence of Young measures associated with the elements of the sequence of simple functions convergent pointwise to $f$. Finally, we illustrate the result with one-dimensional example.

\section[]{Young measures -- necessary information and an auxiliary result}

The first part of this section can serve as a very brief introduction to the theory of Young measures. In the second part we state and prove  lemma mentioned at the end of the Introduction.

\subsection[]{An outline of the Young measures theory}

We gather now some information about Young measures. An interested reader can find details, together with proofs and further bibliography, for example in \cite{Attouch}, \cite{Florescu}, \cite{Gasinski}, \cite{Pedregal}, \cite{Roubicek}. Theorems that plays crucial role in the sequel, namely 
theorem \ref{isomorphism}, theorem \ref{diagram}, corollary \ref{corollaryWarga}, are stated and proved in \cite{Roubicek}.

Let $\mathbb{R}^d\supset\varOmega$ be nonempty, bounded open set and let 
$K\subset\mathbb{R}^l$ be compact. Let $(f_n)$ be a sequence of functions from $\varOmega$ to $K$, convergent to some function $f_0$ weakly$^{\ast}$ in $L^{\infty}$. Finally, let $\varphi$ be an arbitrary continuous real valued function on $\mathbb{R}^l$. Then the sequence $(\varphi(f_n))$ is uniformly bounded in $L^{\infty}$ norm and therefore by the Banach -- Alaoglu theorem there exists a subsequence of $(\varphi(f_n))$   
weakly$^{\ast}$ convergent to some function $g$. In general $g\neq\varphi((f_0))$.
 L. C. Young proved in \cite{Young}, that there exists a subsequence of $(\varphi(f_n))$, not relabelled, and a family $(\nu_x)_{x\in\varOmega}$ of probability measures with supports $\textnormal{supp}\nu_x\subseteq K$, such that 
$\forall\,\varphi\in C(\mathbb{R}^l)\;\forall w\in L^1(\varOmega)$ there holds
\[
\lim\limits_{n\to\infty}\int\limits_{\varOmega}\varphi(f_n(x))w(x)dx=
\int\limits_{\varOmega}\int\limits_K\varphi(s)\nu_x(ds)w(x)dx:=
\int\limits_{\varOmega}\overline{\varphi}(x)w(x)dx.
\]
This family of probability measures is today called a \textit{Young measure associated with the sequence} $(f_n)$. 

In applications there often appears important and particularly simple form of Young measure, a \textit{homogeneous} Young measure. This is 'family' $(\nu_x)_{x\in\varOmega}$ that does not depend on the variable $x$.

In 1989 J. M. Ball proved the following theorem.  Let $\varOmega$ be a measurable subset of $\mathbb{R}^d$, $v\colon [0,+\infty)\to [0,+\infty)$ a continuous, nondecreasing function such that $\lim\limits_{t\to\infty}v(t)=+\infty$. By $\psi$ we denote a function 
$\psi\colon\varOmega\times\mathbb{R}^l\ni(x,\lambda)\to\psi(x,\lambda)\in
\overline{\mathbb{R}}$ satisfying Carath\'eodory conditions: it is measurable with respect to the first, and continuous with respect to the second variable. Consider further a sequence $(f_n)$ of functions on $\varOmega$ with values in $\mathbb{R}^l$, satisfying the condition
\[
\sup\limits_n\int\limits_{\varOmega}v(\vert f_n(x)\vert)dx<+\infty.
\]
\begin{theorem}
(\cite{Ball}) Under the above assumptions, there exists a subsequence of $(f_n)$, not relabelled, and a family $(\nu_x)_{x\in\varOmega}$ of probability measures, dependent measurably on $x$, such that if for any Carath\'eodory function $\psi$ the sequence 
$(\psi(x,f_n(x))$ is weakly convergent in $L^1(\varOmega)$, then its weak limit is a function
\[
\overline{\psi}(x)=\int\limits_{\mathbb{R}^l}\psi(x,\lambda)d\nu_x(\lambda).
\]
\end{theorem}

We now turn our attention to the presentation of the Young measures as in \cite{Roubicek}. In general, Young measures can be looked at as the element of the space conjugate to the space $L^1(\varOmega,C(K))$ of Bochner integrable functions on $\varOmega\subset\mathbb{R}^d$ with values in $C(K)$. The space $L^1(\varOmega,C(K))$ is isometrically isomorphic to the space $Car(\varOmega,K;\mathbb{R})$ of the Carath\'eodory functions, equipped with the norm
\[
\|h\|_{Car}:=\int\limits_{\varOmega}\sup\limits_{k\in K}\vert h(x,k)\vert dx.
\]
Let $h\in L^1(\varOmega,C(K))$. Denote by $\mathcal{U}$ the set of all measurable functions on $\varOmega$ with values in $K$. Consider a mapping 
\[
i\colon \mathcal{U}\rightarrow L^1(\varOmega ,C(K))^\ast 
\]
defined by the formula
\[
\langle i(f),h\rangle :=\int\limits_{\varOmega}h(x,f(x))dx.
\]
By $Y(\varOmega ,K)$ we denote the weak$^\ast$ closure of
the set $i(\mathcal{U})$ in $L^1(\varOmega ,C(K))^\ast$:
\[
Y(\varOmega ,K):=\Bigl\{L^1(\varOmega ,C(K))^\ast\ni\eta :\exists 
(f_n)\subset \mathcal{U}:i(f_n)\xrightarrow[n\rightarrow\infty]{w^\ast}\eta\Bigr\}.
\]
Denote by
\begin{itemize}
\item
$rca(K)$ -- the space of regular, countably additive scalar measures on $K$, equipped with the norm $\|m\|_{rca(K)}:=\vert m\vert(\varOmega)$, where $\vert\cdot\vert$ stands in this case for the total variation of the measure $m$. With this norm $rca(K)$ is a Banach space;
\item
$rca^1(K)$ -- the subset of $rca(K)$ with elements being probability measures on $K$;
\item
$L_{w^\ast}^{\infty}(\varOmega ,\textnormal{rca}(K))$ -- the set of the weakly$^\ast$ measurable mappings
\[
\nu\colon\varOmega\ni x\to\nu(x)\in rca(K).
\]
We equip this set with the norm
\[
\Vert\nu\Vert_{L_{w^\ast}^{\infty}(\varOmega ,\textnormal{rca}(K))}:=
\textnormal{ess}\sup\bigl\{\Vert\nu (x)\Vert_{\textnormal{rca}(K)}:
x\in\varOmega\bigr\}.
\]
By the Dunford -- Pettis theorem this space is isometrically isomorphic with the space
$L^1(\varOmega ,C(K))^\ast$.
\end{itemize}
Now define an element $\eta$ of $L^1(\varOmega ,C(K))^\ast$ by the formula
\[
\eta\colon L^1(\varOmega ,C(K))\ni h\to\langle\eta ,h\rangle :=\int\limits_{\varOmega}
\Bigl(\int\limits_{K}
h(x,k)d\nu_x(k)\Bigr)dx,
\]
which in turn will be the value of the mapping
\[
\psi\colon L^{\infty}_{w^\ast}(\varOmega,\textnormal{rca}(K))\ni
\nu\to\psi (\nu ):=\eta\in L^1(\varOmega ,C(K))^\ast
\]
\begin{theorem}\label{isomorphism}
The mapping $\psi$ defined above is an isometric isomorphism bet\-ween the spaces
$L^{\infty}_{w^\ast}(\varOmega,\textnormal{rca}(K))$ and $L^1(\varOmega ,C(K))^\ast$.
\end{theorem}
The set of the Young measures on the compact set $K\subset\mathbb{R}^l$ will be denoted by $\mathcal{Y}(\varOmega ,K)$:
\[
\mathcal{Y}(\varOmega ,K):=\bigl\{\nu=(\nu (x))\in
L^{\infty}_{w^\ast}(\varOmega,\textnormal{rca}(K)):\nu_x\in
\textnormal{rca}^1(K)\;\textnormal{for a.a }x\in\varOmega\bigr\}.
\]
We will write $\nu_x$ or $(\nu_x)_{x\in\varOmega}$ instead of $\nu (x)$.

Finally, we define the Dirac mapping $\delta$: $\forall x\in\varOmega$
\begin{equation} \label{delta}
\delta\colon\mathcal{U}\ni f\to[\delta(f)](x):= \delta_{f(x)}\in
\mathcal{Y}(\varOmega ,K)
\end{equation}
\begin{theorem} \label{diagram}
The diagram
\[
\begin{xy}
{\ar@{->} (0,0)*+{\mathcal{U}}; (-25,-20)*+{\mathcal{Y}(\varOmega;K)}}?*!/^3mm/{\delta};
{\ar@{->} (0,0)*+{\mathcal{U}}; (25,-20)*+{Y(\varOmega;K)}}?*!/_3mm/{i};
{\ar@{<->} (-25,-20)*+{\mathcal{Y}(\varOmega;K)}; (25,-20)*+{Y(\varOmega;K)}}?*!/^5mm/{\psi};
\end{xy}
\]
is commutative.
\end{theorem}
This means, that for any $f\in\mathcal{U}$ there exists a Young measure associated with it.
\begin{corollary} \label{corollaryWarga}
The set $\mathcal{Y}(\varOmega;K)$ of all Young measures is a convex, compact, sequentially compact set in which $\delta(U)$ is dense.
\end{corollary}

\subsection[]{An auxiliary lemma}

Denote by $\mu$ the normalized Lebesgue measure on a nonempty, bounded subset $\varOmega$ of $\mathbb{R}^d$ with positive measure $M$: $d\mu(x):=\tfrac 1 Mdx$ with a $d$-dimensional Lebesgue measure $dx$. Let $\{\varOmega\}_{i=1}^n$ be a partition of $\varOmega$ into open, pairwise disjoint subsets $\varOmega_i$ with Lebesgue measure $m_i>0$, such that $\bigcup\limits_{i=1}^ncl(\varOmega_i)=cl(\varOmega)$, where '$cl$' stands for 'closure'. By $\textbf{1}_A$ we denote the characteristic function of the set $A$.
\begin{theorem}(see e.g. \cite{Puchala})\label{Explicit_Ja}
Choose and fix points $p_i\in\mathbb{R}^l$, $i=1,2,\dots,n$, and let $f$
be a simple function:
\[
f:=\sum\limits_{i=1}^np_i\textbf{1}_{\varOmega_i}. 
\]
Then the Young measure associated with $f$ is of the form
\[
\nu_x=\frac 1 M\sum\limits_{i=1}^nm_i\delta_{p_i}.
\]
\end{theorem}
\begin{remark}
Observe that in this case $\nu_x$ is a homogeneous Young measure.
\end{remark}
\begin{definition}
The Young measure associated with simple function will be called a 
\textnormal{simple Young measure}.
\end{definition}
We now recall the notion of weak$^\ast$ convergence of measures on compact sets.
\begin{definition}
We say that a sequence $(\nu_n)$ of bounded measures on a compact set $K\subset\mathbb{R}^l$ converges weakly$^\ast$ to a measure $\nu_0$, if 
$\forall\beta\in C(K,\mathbb{R})$ there holds
\[
\lim\limits_{n\to\infty}\int\limits_K\beta(k)d\nu_n(k)=\int\limits_K\beta(k)d\nu_0(k).
\]
\end{definition}
We now prove a useful lemma.
\begin{lemma}\label{aux_lemma}
Let $f\colon\varOmega\to K$ be a measurable function and let $(f_n)$ be a pointwise convergent to $f$ sequence of simple functions. Then the Young measure $\nu^f$ associated with $f$ is a weak$^\ast$ limit of the sequence of the simple Young measures associated with respective elements of $(f_n)$.
\end{lemma}
\begin{proof}
 Choose and fix $\varepsilon>0$. Using change of variable theorem, continuity of the function $\beta$ and the finiteness of the measure of $\varOmega$, we infer the existence of $n_0\in\mathbb{N}$ such that $\forall\, m,n>n_0$ we have
\begin{multline}
\Bigl\vert\int\limits_K\beta(k)d\nu_n-\int\limits_K\beta(k)d\nu_m\Bigr\vert=
\Bigl\vert\int\limits_\varOmega\beta(f_n(x))d\mu-
\int\limits_\varOmega\beta(f_m(x))d\mu\Bigr\vert\leq \\ \notag
\leq\int\limits_\varOmega\vert\beta(f_n(x))-\beta(f_m(x))\vert d\mu\leq\varepsilon\cdot\mu(\varOmega).
\end{multline}
This means that $(\nu_n)$ is a weak$^\ast$ Cauchy sequence in $\mathcal{Y}(\varOmega;K)$, so by corollary \ref{corollaryWarga} there exists a weak$^\ast$ limit $\rho=(\rho_x)_{x\in\varOmega}$ of $(\nu_n)\in\mathcal{Y}(\varOmega;K)$. 

By the equation (\ref{delta}) $\nu^f=\delta_{f(x)}$. Now choose and fix $x_0\in\varOmega$. We then have
\begin{multline}
\Bigl\vert\int\limits_K\beta(k)d\rho_{x_0}-
\int\limits_K\beta(k)d\delta_{f(x_0)}\Bigr\vert\leq 
\Bigl\vert\int\limits_K\beta(k)d\rho_{x_0}-\int\limits_K\beta(k)d\nu_n\Bigr\vert+\\ \notag
+\Bigl\vert\int\limits_K\beta(k)d\nu_n-\int\limits_K\beta(k)d\delta_{f(x_0)}\Bigr\vert.
\end{multline}

The first term on the right-hand side is arbitrarily small since $\rho$ is a weak$^\ast$ limit of $(\nu_n)$. For the second term 
\begin{multline}
\Bigl\vert\int\limits_K\beta(k)d\nu_n-\int\limits_K\beta(k)d\delta_{f(x_0)}\Bigr\vert\leq
\Bigl\vert\int\limits_K\beta(k)d\nu_n-\int\limits_K\beta(k)d\delta_{f_n(x_0)}\Bigr\vert+\\ \notag
+\Bigl\vert\int\limits_K\beta(k)d\delta_{f_n(x_0)}-
\int\limits_K\beta(k)d\delta_{f(x_0)}\Bigr\vert.
\end{multline}
The first term above vanishes because $\nu_n$ is a homogeneous Young measure associated with the simple function $f_n$, $n\in\mathbb{N}$. The second term tends to $0$ as $n\to\infty$ for the sequence $(f_n)$ converges pointwise to $f$. 
\end{proof}
\begin{remark}
\begin{itemize}
\item[(i)]
observe that $\nu^f$ need not be a homogeneous Young measure;
\item[(ii)]
for more general versions of the above result see \cite{Balder1} and references cited there.
\end{itemize}
\end{remark}
\begin{corollary}
The set of all simple Young measures is weak$^\ast$ dense in the set of the Young measures associated with functions from $\mathcal{U}$.
\end{corollary}
\section[]{Some necessary notions from probability theory and notation}
To set up notation, we recall now standard probabilistic notions needed in the sequel. If $\Sigma$ is a $\sigma$-algebra of subsets of a nonempty set $A$ and $P$ -- a measure on $\Sigma$, then the triple $(A,\,\Sigma,\, P)$ is called a measure space, and a \textit{probability space} if $P$ is a probability measure. A \textit{random va\-riable} (or a \textit{random vector}) $X\colon A\to\mathbb R^d$ is a function such that for any Borel set 
$B\subseteq\mathbb R^d$ there holds $X^{-1}(B)\in\Sigma$. Obviously, if $\varphi\colon\mathbb R^d\to\mathbb R^l$ is a Borel function, then $\varphi(X)$ is a random variable. The \textit{probability distribution} on $\mathbb R^d$ is any probability measure $P$ on the $\sigma$-algebra $\mathcal B(\mathbb R^d)$ of Borel subsets of $\mathbb R^d$. The probability distribution of a random variable $X$ with values in $\mathbb R^d$ is a probability measure $P_X$ on $\mathbb R^d$ defined for any $B\in\mathcal B(\mathbb R^d)$ by the equality $P_X(B):=P(X^{-1}(B))$. Consequently, for the  distribution of the random variable $\varphi(X)$ we have: 
for any $C\in\mathcal B(\mathbb R^l)$
\[
P_{\varphi(X)}(C)=P(\varphi(X)^{-1}(C))=P_X(\varphi^{-1}(C)).
\] 
If $P$ is a probability distribution on $\mathbb R^d$ and for some Lebesgue integrable function $g\colon\mathbb R^d\to\mathbb R$ there holds:
$\forall\, A\in\mathcal B(\mathbb R^d)\; P(A)=\int\limits_Ag(x)dx$, then the function $g$ is called a \textit{density} of $P$.

Let $\varOmega$ be a Borel subset of $\mathbb R^d$ with Lebesgue measure $M>0$. We say that random variable 
$U\colon\mathbb R^d\to\mathbb R^l$ is \textit{uniform} on $\varOmega$, if its density $g_u$ is of the form
\[
g_U(x)=
\begin{cases}
\tfrac 1 M, & x\in\varOmega\\
0, & x\notin\varOmega.
\end{cases}
\]
The probability distribution $P_U$ is then called the \textit{uniform distribution}.
\section[]{Main result}
As in the previous sections, let $\varOmega$ be an open subset of $\mathbb R^d$ with Lebesgue measure $M>0$, $d\mu(x) :=\tfrac{1}{M}dx$, where $dx$ is the $d$ -- 
dimensional Lebesgue measure on $\varOmega$ and let $K\subset\mathbb R^l$ be compact.
Denote $P:=\tfrac 1 Mdx$.

Finally, we are ready to formulate the main theorem of the article.
\begin{theorem}\label{main}
Let $f\colon\mathbb R^d\supset\varOmega\rightarrow K\subset\mathbb R^l$ be a Borel function with Young measure $\mu^f$. Then $\mu^f$ is the probability distribution of the random variable $Y=f(U)$, where $U$ has a uniform distribution on $\varOmega$.
\end{theorem}
\begin{proof}
The distribution of a random variable $Y$ is of the form:
$\forall\,C\in\mathcal B(K)$, $P_{f(U)}(C)=P_U(f^{-1}(C))$. Let $f$ be constant on $\varOmega$ with value $p$ and vanish on the complement of $\varOmega$. By theorem
\ref{Explicit_Ja} we have
$\mu^f=\delta_p$. For any $C\subseteq K$ we have
\[
\mu^f(C)=\int\limits_{\mathbb R^l}\textbf 1_C(p)d\delta_p=
\begin{cases}
1, & p\in C\\
0, & p\notin C.
\end{cases}
\]
On the other hand,
\[
P_U(f^{-1}(C))=\int\limits_{f^{-1}(C)}g_udP=\frac 1 M\int\limits_{\{x:f(x)\in C\}}dx=
\begin{cases}
\tfrac 1 M\cdot M=1, & p\in C\\
\tfrac 1 M\cdot 0=0, & p\notin C.
\end{cases}
\]
Thus $P_Y=\mu^f$. This equality also holds when $f$ is a simple function, due to the linearity of the integral. Since functions under consideration have values in the compact set $K$, lemma \ref{aux_lemma} and the dominated convergence theorem yields the result for any Borel $f$.
\end{proof}
\section[]{Some applications and comments}

Theorem  \ref{main} provides direct link between the  Young measure basic concepts and the probability theory. This allows a wealth of probabilistic tools to be used to derive the explicit forms of the density functions of Young measures in many practically interesting  cases. 
To illustrate this point let us consider the following problem.

Let $\varOmega$ be as at the beginning of the previous section. Consider $\{\varOmega\}$ -- an open partition of $\varOmega$ into at most countable number of open subsets $\varOmega_1,\varOmega_2,\dots,\varOmega_n,\dots$ such that
\begin{itemize}
\item[(i)]
the elements of $\{\varOmega\}$ are pairwise disjoint;
\item[(ii)]
$\bigcup\limits_{i}\overline{\varOmega}_i=\overline{\varOmega}$, where $\overline{A}$ denotes the closure of the set $A$.
\end{itemize}

Let us consider   functions $f_i\colon\varOmega_i\to K\subset\mathbb{R}^d$, $i=1,2,...$, with inverses $f_i^{-1}$  that are  continuously differentiable on $f(\varOmega_i)$ and let $K_i:=\overline{f(\varOmega_i)}$ be compact. Denote for each  $i=1,2,...$ the Jacobian matrix of $f_i^{-1}$ by $J_{f_i^{-1}}$.

Let a  function $f\colon\varOmega\to K$, with $K:=\overline{f(\varOmega)}$ compact, be such that
\begin{equation}\label{uncountable}
f(x)=\sum\limits_{i} f_i(x)\chi_{\varOmega_i}(x)\ ,\ x\in \bigcup\limits_{i}\varOmega_i . 
\end{equation}

Then the following result holds.
\begin{proposition}\label{Prop}
The Young measure associated with Borel function $f$ satisfying (\ref{uncountable}) is a homogeneous one and its density $g$  with respect to the Lebesgue measure on $K$ is of the following form
\begin{equation}\label{density}
g(y)=\frac 1 M\sum\limits_{i:y\in\varOmega_i }\vert J_{f_i^{-1}}(y)\vert
\end{equation}
\end{proposition}
The above result is a conclusion of  our main Theorem  \ref{main} and the general probabilistic results concerning the distributions of the functions of random vectors/variables, (compare \cite{Soong} or the classial work of Hoog and Craig \cite{Hoog} ).  

Observe, that the above Proposition \ref{Prop}  extends the results stated in Propositions 3.2, 3.4 in\cite{Puchala} which are proven there directly on the basis of  Young measure notion. 

As an example of a more specific application of Proposition \ref{Prop} let us consider the one-dimensional case ($d=1$) and the following function $f\colon]0,1[\to ]0,1[$:

\begin{equation}\label{specificB}
f(x)=\sum\limits_{i=2}^\infty (n x - 1)\chi_{\big[\frac{1}{n},\frac{1}{n-1}\big[}(x) 
\end{equation}

The function given by (\ref{specificB}) is of the form (\ref{uncountable}), thus by the Proposition \ref{Prop}, the Young measure associated with this  $f$  has the density function $g$ that is a piecewise constant one and given by the following formula:
\begin{equation}\label{specificYM}
g(y)=(H_n-1)\chi_{\big[\frac{1}{n},\frac{1}{n-1}\big[}(y)\ ,\ y\in  ]0,1[ 
\end{equation}
where $H_n$  stands for the $n$-th harmonic number. 

Due to the main Theorem \ref{main} various other results e.g. related to Borel functions  with different  dimensions of their domains and  images,  can be obtained with the help of other known probabilistic results concerning distributions of functions of random variables or vectors. Such results are  of particular importance in the engineering  practice, as they allow us the determination of specific values of Young's functionals, either directly or, in the more complex cases, by the Monte Carlo simulation

(for the latter see \cite{AZG_PP_1} and \cite{AZG_PP_2}).

\end{document}